\documentclass{article}
\usepackage{amsmath,amssymb,amsthm}
\usepackage{cite}
\usepackage{mathtools}
\usepackage{tikz}
\usetikzlibrary{patterns}
\usepackage{pgfplots}
\usepackage{pgfplotstable}
\usepackage{geometry}
\usepackage{soul}
\usepackage{color}
\usepackage{booktabs}
\usepackage{mathrsfs}
\usepackage{cleveref}
\usepackage[utf8]{inputenc}
\usepackage{enumitem}
\usepackage{algorithm}
\usepackage{algpseudocode}

\newcommand{\rank}{\hbox{rk}}

\newcommand{\cR}{\ensuremath{\mathcal{R}}}
\newcommand{\cK}{\ensuremath{\mathcal{K}}}

\newcommand{\smb}{\left[\begin{smallmatrix}}
\newcommand{\sme}{\end{smallmatrix}\right]}

\newcommand\R{\ensuremath{\mathbb{R}}}

\title{Low-rank updates and  divide-and-conquer methods for quadratic matrix equations}

\author{%
	Daniel Kressner\thanks{EPF Lausanne, Switzerland,
		{\tt daniel.kressner@epfl.ch}} 
	\and
	Patrick K\"urschner\thanks{KU Leuven, Electrical Engineering (ESAT), Kulak Kortrijk Campus, Belgium, {\tt patrick.kurschner@kuleuven.be}}
	\and
	Stefano Massei\thanks{EPF Lausanne, Switzerland,
		{\tt stefano.massei@epfl.ch}} 
}
\date{}
\pgfplotsset{compat=1.9}

\pgfplotstableset{
	every head row/.style={before row=\toprule,after row=\midrule},
	clear infinite
}

\usepackage{amsopn}
\DeclarePairedDelimiter{\norm}{\lVert}{\rVert}
\DeclarePairedDelimiter{\abs}{\lvert}{\rvert}
\DeclareMathOperator{\diag}{diag}
\DeclareMathOperator{\tridiag}{tridiag}

\renewcommand{\leq}{\leqslant}
\renewcommand{\geq}{\geqslant}

\newtheorem{proposition}{Proposition}[section]
\newtheorem{lemma}[proposition]{Lemma}
\newtheorem{theorem}[proposition]{Theorem}

\newtheorem{definition}[proposition]{Definition}

\newtheorem{example}[proposition]{Example}

\begin{document}
	\maketitle
	\begin{abstract}
		In this work, we consider two types of large-scale quadratic matrix equations: Continuous-time algebraic Riccati equations, which play a central role in optimal and robust control, and unilateral quadratic matrix equations, which arise from stochastic processes on 2D lattices and vibrating systems. We propose a simple and fast way to update the solution to such  matrix equations under low-rank modifications of the coefficients. Based on this procedure, we develop a divide-and-conquer method for quadratic matrix equations with coefficients that feature a specific type of hierarchical low-rank structure, which includes banded matrices. This generalizes earlier work on linear matrix equations. Numerical experiments indicate the advantages of our newly proposed method versus iterative schemes combined with hierarchical low-rank arithmetic.
	\end{abstract}
	
	\section{Introduction}\label{sec:intro}
	
	This paper is concerned with numerical algorithms for treating two types of quadratic matrix equations with large-scale, data-sparse coefficients.
	
	\begin{paragraph}{\bf Type 1: CARE.}
		A \emph{continuous-time algebraic Riccati equation (CARE)} takes the form
		\begin{equation}\label{eq:care}
		A^* X E + E^* X A - E^*XFXE + Q = 0,
		\end{equation}
		where $A,E,F,Q$ are real $n\times n$ matrices, such that $E$ is invertible and $F,Q$ are symmetric positive semi-definite.
		Motivated by its central role in robust and optimal control~\cite{Rus79,Lancaster1995,Sima1996,Loc01,Benner2004}, this class of equations has been widely studied in the literature; see, e.g.,~\cite{BinIM12,Benner2013,BenBKetal18a}.  A solution $X$ to~\eqref{eq:care} is called stabilizing if the so called closed-loop matrix $A-FXE$ is stable, that is, all its eigenvalues are contained in the open left half plane. 
		Mild conditions on the coefficients (see, e.g.,~\cite[Sec. 2.2.2]{BinIM12}) ensure the existence, uniqueness, and symmetric positive semi-definiteness of such a stabilizing solution $X$.
		
		We consider the case when $n$ is large and $F$ has low rank, that is, $F=BB^*$ for some matrix $B\in\mathbb R^{n\times m}$ with $m\ll n$. This is a common assumption in linear-quadratic optimal control problems, where $m$ corresponds to the number of inputs~\cite{BenLP08,Benner2013}. However, we do not impose low rank on $Q$, which allows for having a large number of outputs in control problems, e.g., when observing the state directly. To simplify the exposition, we will focus the discussion mostly on the case $E=I$; the extension to general invertible $E$ will be explained in Section~\ref{sec:gCARE}. 
		
		For $F=0$, the equation~\eqref{eq:care} becomes linear and is called Lyapunov equation. A low-rank updating procedure for such linear matrix equations has been proposed recently in~\cite{kressner2017low}. In this work, we extend this procedure to CARE. More specifically, assuming that $X_0$ satisfies a reference CARE
		\begin{equation}\label{eq:ref-care}
		A_0^* X_0  +  X_0 A_0 - X_0F_0X_0 + Q_0 = 0,
		\end{equation}  
		we aim at computing a correction $\delta X$ such that $X:=X_0+\delta X$ solves the modified CARE
		\begin{equation}\label{eq:pert-care}
		\begin{split}
		A^* X + X A - X F X + Q = 0,
		\end{split}
		\end{equation}
		with
		\[
		A:= A_0+\delta A, \quad F:= F_0+\delta F, \quad Q:= Q_0+\delta Q, \quad \text{with $\delta A, \delta F$, $\delta Q$ of low rank.}
		\]
		Subtracting~\eqref{eq:ref-care} from~\eqref{eq:pert-care} yields
		\begin{equation}\label{eq:cor-care}
		\begin{split}
		( A- FX_0)^* \delta X  + \delta X (A- FX_0) - \delta XF\delta X +  \widehat Q= 0.
		\end{split}
		\end{equation}
		The modified constant term $\widehat Q:= \delta Q +\delta A^*X_0
		+ X_0\delta A- X_0\delta FX_0$ satisfies \[ \rank(\widehat Q) \le  \rank(\delta Q) + 2\rank(\delta A)+\rank(\delta F), \]
		where $\rank(\cdot)$ denotes the rank of a matrix. Hence, independently of the rank of $Q_0$, the constant term of the CARE~\eqref{eq:cor-care} is guaranteed to have low rank.
		%Notice that, in order to retrieve these settings, the low rank assumption on the matrix $F_0$ is not avoidable.
		Note that most algorithms for large-scale Riccati equations~\cite{BenLP08,Benner2013,simoncini2013two,Sim16,BenBKetal18} assume the constant term to be of low rank which, in turn, may render them unsuitable for solving~\eqref{eq:pert-care}. In contrast, the formulation~\eqref{eq:cor-care} is well suited for such methods, returning an approximation of $\delta X$ in the form of a symmetric low-rank factorization.
		%If, in addition, the matrix $A$ linear coefficients of \eqref{eq:cor-care} allow for efficient matrix-vector multiplication and solution of --- possibly shifted --- linear systems, then approximating $\delta X$ is feasible for large dimension $n$.
	\end{paragraph}
	
	\begin{paragraph}{\bf Type 2: UQME.}
		A \emph{unilateral quadratic matrix equation (UQME)} takes the form
		\begin{equation}\label{eq:uqme}
		AX^2+BX+C=0,
		\end{equation}
		with $A,B,C\in\mathbb R^{n\times n}$. The spectrum of a solution to~\eqref{eq:uqme}
		%--- whether they exist ---
		corresponds to a subset of the $2n$ eigenvalues of the matrix polynomial
		\begin{equation}\label{eq:phi}
		\varphi(\lambda):= 
		\lambda^2A+\lambda B + C.
		\end{equation}
		Instances of equation~\eqref{eq:uqme} arise in overdamped systems in structural mechanics \cite{GuoHyp} and are at the core of the matrix analytic method for \emph{quasi-birth--death} (QBD) stochastic processes \cite{BiniMarkov}.
		
		A typical situation in applications is that the eigenvalues of $\varphi(\lambda)$ are separated by the unit circle into two subsets of cardinality $n$:
		\begin{equation}\label{eq:split}
		|\lambda_1|\leq \dots |\lambda_n|\leq 1\leq |\lambda_{n+1}|\leq\dots\leq|\lambda_{2n}|,\qquad |\lambda_n|<|\lambda_{n+1}|,
		\end{equation}
		and it is of interest to compute the \emph{minimal solution} of~\eqref{eq:uqme}, that is, the solution $X$ associated with $\lambda_1,\dots,\lambda_n$. Note that some of the eigenvalues are allowed to be infinite.

		\eqref{eq:split} implies  that the matrix $X$ is the only power bounded solution of \eqref{eq:uqme}; this uniquely identify the \emph{matrix-geometric property} \cite{BiniMarkov} of certain QBD processes. \eqref{eq:split} also guarantees the quadratic convergence of the \emph{cyclic reduction} algorithm for computing $X$ \cite[Theorem 9]{Bini2009}. The minimal solution  can be constructed as $X=V\diag(\lambda_1,\dots,\lambda_n)V^{-1}$ if the matrix $V$ containing the eigenvectors associated with $\lambda_1,\dots,\lambda_n$ is invertible \cite{Higham2000b}. 
		This property is usually met in practice and in the QBD setting  it can be ensured via the probabilistic interpretation of the minimal solution \cite[Section 6.2]{Latouche}. 
		
		We assume that the minimal solution exists and that \eqref{eq:split} holds for a reference equation
		\begin{equation}\label{eq:ref-uqme}
		A_0 X_0^2  +  B_0X_0 + C_0 = 0,
		\end{equation}  
		as well as for the modified equation
		\begin{equation}\label{eq:pert-uqme}\begin{split}
		(A_0+\delta A) (X_0+\delta X)^2  +  (B_0+\delta B)(X_0+\delta X)  
		+ (C_0+\delta C) = 0,
		\end{split}
		\end{equation}
		where $\delta A,\delta B$ and $\delta C$ are given  low-rank matrices.
		%Once again, we aim to compute $\delta X$ which express the difference between the minimal solutions of \eqref{eq:pert-uqme} and \eqref{eq:ref-uqme}.
		
		Denoting $A:=A_0+\delta A,B:=B_0+\delta B,C:=C_0+\delta C$ and subtracting   \eqref{eq:ref-uqme} from \eqref{eq:pert-uqme} yields the following equation for the correction $\delta X$:
		\begin{equation}\label{eq:cor-uqme}
		A\delta X^2 + (AX_0+B)\delta X + A\delta X X_0+\widehat C=0 ,
		\end{equation}
		where $\widehat C:=\delta A X_0^2+\delta BX_0+\delta C$ has rank bounded by $\rank(\delta A)+\rank(\delta B)+\rank(\delta C)$.
		Note that equation~\eqref{eq:cor-uqme}  is not a UQME. Nevertheless, as will be seen in Section~\ref{sec:correction}, there is still a  correspondence between the solutions of~\eqref{eq:cor-uqme} and an appropriately chosen eigenvalue problem. Similarly as for CARE, the low rank of $\widehat C$ will allow us to devise an efficient numerical method for~\eqref{eq:cor-uqme}.
		%This shed some lights on the low-rank approximability of $\delta X$.
		%The computation of $\delta X$ is addressed by means of a projection method that we propose in Section~\ref{sec:update}.
	\end{paragraph}
	
	\begin{paragraph}{\bf Quadratic matrix equations with hierarchical low-rank structure.}
		In the second part of the paper we focus on quadratic  equations with coefficients that feature hierarchically low-rank structure. More specifically, the coefficients of a CARE~\eqref{eq:care} or a UQME~\eqref{eq:uqme} are assumed to be \emph{hierarchically off-diagonal low-rank (HODLR) matrices}\cite{Ambikasaran2013,Hackbusch2015}. This framework aligns well with the low-rank updates discussed above, because HODLR matrices are block diagonalized by a low-rank perturbation and, in turn, the corresponding reference equations~\eqref{eq:ref-care} and~\eqref{eq:ref-uqme} decouple into two equations of smaller size. 
		Applying this idea recursively results in a divide-and-conquer method for solving UQMEs with HODLR coefficients and CAREs with a low-rank  quadratic term and all other coefficients in the HODLR format. 
		
		Existing fast algorithms that address such (and more general) scenarios are based on combining a matrix iteration with fast arithmetic in hierarchical low-rank format. For CAREs, a combination of the sign function iteration with hierarchical matrices has been proposed in~\cite{Baur2006,Grasedyck2003a}. For UQMEs, a combination based on cyclic reduction has been proposed in~\cite{Bini2016,Bini2017}. As pointed out in~\cite{kressner2017low}, a disadvantage of these strategies is that they exploit the structure only indirectly and rely on repeated recompression during the iteration, which may constitute a computational bottleneck.
	\end{paragraph}
	
	\begin{paragraph}{\bf Outline.}
		
		The rest of this paper is organized as follows. In Section~\ref{sec:correction}, we study the correction equations~\eqref{eq:cor-care} and~\eqref{eq:cor-uqme}, with a particular focus on providing intuition why one can expect their solutions to admit good low-rank approximations.  Section~\ref{sec:update} is concerned with numerical methods for obtaining such low-rank approximations. While a variety of large-scale solution methods have been recently developed for~\eqref{eq:cor-care}, the equation~\eqref{eq:cor-uqme} is non-standard and requires the development of a novel large-scale solver, which may be of independent interest. Section~\ref{sec:dac} utilizes these solvers to derive divide-and-conquer methods for CARE and UQME featuring HODLR matrix coefficients. Finally, Section~\ref{sec:numerical} highlights several applications of these divide-and-conquer methods and provides numerical evidence of their effectiveness.
		
	\end{paragraph}

	\section{Analysis of the correction equations}\label{sec:correction}
	The purpose of this section is to study properties of the correction matrix $\delta X$, which satisfies one of the two correction equations,~\eqref{eq:cor-care} or~\eqref{eq:cor-uqme}. 
	
	%Concerning its low-rank approximability, we say that $\delta X$ admits an \emph{$\epsilon$ approximation of rank $k$} if there exists a matrix $Y$ of rank at most $k$ such that $\norm{\delta X-Y}_2\le \epsilon$. This is equivalent to say that the $(k+1)$-th largest singular value of $\delta X$ is less than $\epsilon$. 
	
	\subsection{Existence and low-rank approximability}\label{sec:low-rank}
	
	A necessary requirement of most solvers for large-scale matrix equations to perform well is that the solution admits good low-rank approximations.
	This property can sometimes be verified a priori by showing that the singular values exhibit a strong decay. In the following we first recall such results for linear matrix equations and then use them to shed some insight on the low-rank approximability of $\delta X$. 
	\begin{paragraph}{Singular value decay for linear matrix equations}
		Let us consider the so called \emph{Sylvester equation} % \cite{Simoncini2016}
		\[
		AX+XB=Q
		\]
		with the coefficients $A,B,Q\in\mathbb R^{n\times n}$ such that the spectra of $A$ and $-B$ are disjoint, and $Q$ has rank $k \ll n$.
		Moreover, let $\mathcal R_{h,h}$ denotes the set of rational functions with numerator and denominator degrees at most $h$. 
		
		In \cite{Beckermann2016}, it is shown that for every $r\in\mathcal R_{h,h}$ there exists a  matrix $\widetilde X$ of rank at most $kh$ such that $X-\widetilde X=r(A)Xr(-B)^{-1}$, provided that the right-hand side is well defined. Using that the $(kh+1)$th singular value, denoted by $\sigma_{kh+1}(\cdot)$, governs the $2$-norm error of the best approximation by a matrix of rank at most $kh$, one obtains
		\[
		\sigma_{kh+1}(X)\le \norm{r(A)}_2\norm{r(-B)^{-1}}_2\norm{X}_2.
		\]
		Combined with norm  estimates for rational matrix functions, this leads to the following theorem.
		\begin{theorem}[Theorem 2.1 in \cite{Beckermann2016}]\label{thm:zol}
			Consider the Sylvester equation $AX+XB=Q$, with $Q$ of rank $k$, and let 
			$E$ and $F$ be disjoint compact sets in the complex plane. 
			\begin{itemize}
				\item[$(i)$] If $E,F$ contain
				the numerical ranges of $A$ and $-B$, respectively, then
				\[
				\frac{\sigma_{kh+1}(X)}{\norm{ X}_2}\leq  K_C\,\min_{r\in\mathcal R_{h,h}}\frac{\max_E \abs{r(z)}}{\min_F \abs{r(z)}},
				\]
				where $K_C = 1$ if $A,B$ are normal matrices and $1\le K_C \le (1+\sqrt{2})^2$ otherwise.
				\item[$(ii)$] If $A,B$ are diagonalizable and $E,F$ contain
				the spectra of $A$ and $-B$, respectively, then
				\[
				\frac{\sigma_{kh+1}(X)}{\norm{ X}_2}\leq  \kappa_{\mathsf{eig}}(A)\kappa_{\mathsf{eig}}(B)\,\min_{r\in\mathcal R_{h,h}}\frac{\max_E \abs{r(z)}}{\min_F \abs{r(z)}}
				\]
				where $\kappa_{\mathsf{eig}}(\cdot)$ denotes the $2$-norm condition number of the eigenvector matrix.
			\end{itemize}
		\end{theorem}
		The quantities $Z_h(E,F):=\min_{r\in\mathcal R_{h,h}}\frac{\max_E \abs{r(z)}}{\min_F \abs{r(z)}}$ are known in the literature as \emph{Zolotarev numbers}. When $E$ and $F$ are well separated one can expect that $Z_h(E,F)$ decreases rapidly, as $h$ increases, and quickly  reaches the level of machine precision. Explicit bounds showing exponential decay have been established for various configurations of $E$ and $F$, including disjoint real intervals and circles~\cite{Beckermann2016,Starke1992}.
		% For more general sets, Ganelius in \cite{Ganelius80} showed the inequality
		%\begin{equation}\label{eq:ganelius}
		%Z_h(E,F)\leq \gamma \rho^{-h},\qquad  \rho:=\exp\left(\frac{1}{\capac(E,F)}\right),
		%\end{equation}
		%where the constant $\gamma$ depends only on the geometry of $E$ and $F$ and $\capac(E,F)$ denotes the so called \emph{logarithmic capacity} of the condenser with plates $E$ and $F$.
	\end{paragraph}
	
	\begin{paragraph}{\bf CARE.}
		The existence and uniqueness of a stabilizing solution to the correction equation~\eqref{eq:cor-care} follows immediately from the observation that the closed-loop matrices of CARE~\eqref{eq:pert-care} and CARE~\eqref{eq:cor-care} are identical:
		\[(A-FX_0)-F\delta X=A-FX.\]
		%Hence, if $X$ is a stabilizing solution of~\eqref{eq:pert-care} then it corresponds to the closed loop matrix associated with $X$, that by assumption is stable. This ensures that $\delta X$ is a stabilizing solution, independently on the stabilizing property of $X_0$.
		This yields the following lemma. %, which does not impose a stabilizing property on $X_0$
		\begin{lemma}
			Let $X_0$ be a solution of~\eqref{eq:ref-care}. Then the correction equation~\eqref{eq:cor-care} has a unique stabilizing solution $\delta X$ if and only if the modified equation~\eqref{eq:pert-care} has a unique stabilizing solution $X$.  
		\end{lemma}
		
		To study the low-rank approximability of $\delta X$, let us first assume that $A$ is stable. By 
		rearranging~\eqref{eq:cor-care}, we get
		\[
		A^* \delta X  + \delta XA = -\widehat Q + \delta XF(\delta X  + X_0)+ X_0F\delta X.
		\]
		Hence, $\delta X$ satisfies a Lyapunov equation with the rank of the right-hand side bounded by $2\rank(F)+\rank(\widehat Q)$. If, additionally, the numerical range of $A$ is contained in the open left half plane then the first part of Theorem~\ref{thm:zol} can be applied to yield singular value bounds for $\delta X$. Alternatively, the second part can be applied under the milder assumption that $A$ is diagonalizable.
		
		If $A$ is not stable, we rearrange \eqref{eq:cor-care} as
		\[
		(A-FX)^*\delta X + \delta X(A-FX)=-\widehat Q-\delta XF\delta X. 
		\] 
		As the closed loop matrix $A-FX$ is stable and the rank of the right-hand side is bounded by $\rank(F)+\rank(\widehat Q)$, Theorem~\ref{thm:zol} applies under the assumptions stated above with $A$ replaced by $A-FX$.
		One should note, however, that the obtained bounds are somewhat implicit because they involves the numerical range or the eigenvector conditioning of $A-FX$, quantities that are hard to estimate a priori. If more information is available for the closed loop matrix $A-F\widetilde X$ associated with a stabilizing initial guess $\widetilde X$, one can instead work with the equation
		\[
		(A-F\widetilde X)^*\delta X + \delta X(A-F\widetilde X)=-\widehat Q+\delta XF\delta X-2\delta \widetilde XF\delta \widetilde X,
		\] 
		where $\delta\widetilde X:= \widetilde X-X_0$.
		%and --- by applying Theorem~\ref{thm:zol} --- an explicit bound for the singular values of $\delta X$.
		%It is worth noticing that assuming the stabilizing property of $X_0$ does not play any role in this study of the low-rank approximability of $\delta X$.
	\end{paragraph}
	\begin{paragraph}{\bf UQME.}
		Solutions of the correction equation~\eqref{eq:cor-uqme} are intimately related to the matrix pencil
		\begin{equation}\label{eq:pencil}
		\begin{bmatrix}
		X_0&I\\
		-\widehat C&-(AX_0+B)
		\end{bmatrix}
		-\lambda\begin{bmatrix}
		I & 0 \\0 &A
		\end{bmatrix}.
		\end{equation}
		In fact, a direct computation shows that  $\delta X$ solves \eqref{eq:cor-uqme} if and only if
		\begin{equation}\label{eq:deflating}
		\begin{bmatrix}
		X_0&I\\
		-\widehat C&-(AX_0+B)
		\end{bmatrix}\begin{bmatrix}
		I \\
		\delta X
		\end{bmatrix}=
		\begin{bmatrix}
		I & 0\\ 0 & A
		\end{bmatrix}
		\begin{bmatrix}
		I\\
		\delta X
		\end{bmatrix}X.
		\end{equation}
		For simplicity, let us assume that $A$ is invertible. Then the eigenvalues of~\eqref{eq:pencil} coincide with the eigenvalues of the matrix 
		\[\begin{bmatrix}
		X_0&I\\
		-A^{-1}\widehat C& -(X_0+A^{-1}B)
		\end{bmatrix}.
		\]
		By a similarity transformation,
		\begin{equation} \label{eq:similaritytrafo}
		\begin{bmatrix}
		I & 0 \\
		\delta X&I
		\end{bmatrix}^{-1}
		\begin{bmatrix}
		X_0&I\\
		-A^{-1}\widehat C& -(X_0+A^{-1}B)
		\end{bmatrix}\begin{bmatrix}
		I & 0\\
		\delta X&I
		\end{bmatrix}
		=
		\begin{bmatrix}
		X&I\\
		0 &-(X+A^{-1}B)
		\end{bmatrix}.
		\end{equation}
		Because $X = X_0+\delta X$ is a solution of $\eqref{eq:uqme}$, the quadratic matrix polynomial $\varphi(\lambda)$ defined in~\eqref{eq:phi} admits the factorization
		\[
		\lambda^2 A + \lambda B + C = (\lambda A + AX + B ) (\lambda I - X) = A^{-1} (\lambda I + X + A^{-1}B) (\lambda I - X) .
		\]
		Together with~\eqref{eq:similaritytrafo}, this shows that the eigenvalues of the pencil~\eqref{eq:pencil} coincide with the eigenvalues of $\varphi(\lambda)$. In particular, if $X_0$ and $X$ are the minimal solutions of \eqref{eq:ref-uqme} and \eqref{eq:uqme}, respectively, then the spectra of $X_0$ and $-(X+A^{-1}B)$ are separated by the unit circle. By rearranging \eqref{eq:pert-uqme}, $\delta X$ can be viewed as the solution of a Sylvester equation with these coefficients and low-rank right hand side:
		\begin{equation}\label{eq:sylv-uqme-corr}
		(X+A^{-1}B)\delta X + \delta X X_0= -A^{-1}\widehat C.
		\end{equation}
		This indicates good low-rank approximability of $\delta X$.
	\end{paragraph}
	
	\section{Low-rank updates} \label{sec:update}
	
	In Section~\ref{sec:intro} we already described the basic procedure for updating the solution $X_0$ of a reference CARE or UQME.
	This requires solving correction equations of the form~\eqref{eq:cor-care} or \eqref{eq:cor-uqme}, respectively.
	In the following, we discuss how to solve these correction equations efficiently .
	
	\subsection{Projection subspaces}
	
	According to the discussion in Section~\ref{sec:low-rank}, one may expect that the solutions of~\eqref{eq:cor-care} and~\eqref{eq:cor-uqme} admit good low-rank approximations. A common strategy for obtaining such approximate solutions is to project these matrix equations to a pair of subspaces. To be more specific, let $U,V\in\mathbb R^{n\times t}$ contain orthonormal bases of $t$-dimensional subspaces $\mathcal U,\mathcal V\subset \mathbb R^n$. Then we consider approximate solutions of the form $\widetilde X:=UYV^*$, where $Y\in\mathbb R^{t\times t}$ is obtained
	%imposing a Galerkin condition on the residual with respect to the tensorized space  $\mathcal U\otimes\mathcal V$. This corresponds to require $Y$ to verify a
	from solving a compressed matrix equation. % of dimension $t\times t$.
	
	The choice of the projection subspaces $\mathcal U,\mathcal V$ is key to obtaining good approximations. In the context of matrix equations,
	\emph{rational Krylov subspaces}~\cite{Ruhe1998} are a popular and effective choice.
	\begin{definition} Let $A\in\mathbb R^{n\times n}$, $U_0 \in\mathbb R^{n\times k}$, $k<n$, and $\xi\in\mathbb C^{t}$. The vector space 
		\begin{align*}
		\cR\cK_t(A,U_0,\xi):= \mathrm{range}\Big\{\Big[ U_0,(A-\xi_1 I)^{-1}U_0,\ldots,\Big(\prod\limits_{j=1}^{t-1}(A-\xi_j I)^{-1}\Big)U_0\Big] \Big\}
		\end{align*}
		is called \emph{rational Krylov subspace} with respect to $(A,U_0,\xi)$.
	\end{definition}
	A good choice of shift parameters $\xi_j$ is crucial and we will discuss our choices for CARE and UQME below.
	%, separatel, see the discussion~\cite{Benner2013b,DruS11} for the case of linear equations.
	
	\subsection{Low-rank solution of correction equation for CARE}
	
	To describe our approach for approximating the solution of~\eqref{eq:cor-care}, let us define 
	$A_{\mathsf{corr}}:= A- FX_0$ and suppose that the low-rank updates of the coefficients are given in factorized, symmetry-preserving form:
	\[
	\delta A=U_AV_A^*, \quad \delta
	Q=U_QD_QU_Q^*, \quad \delta F=U_FD_FU_F^*,
	\]
	with $U_A,V_A\in\R^{n\times \rank{(\delta A)}}$, $U_Q\in\R^{n\times \rank{(\delta Q)}}$, $U_F\in\R^{n\times \rank{(\delta F)}}$, and symmetric matrices $D_Q \in\R^{\rank{(\delta Q)}\times \rank{(\delta Q)}}$, $D_F \in\R^{\rank{(\delta F)}\times \rank{(\delta F)}}$. This allows us to write the right-hand side of~\eqref{eq:cor-care} in factorized form $\widehat Q= UDU^*$ as well, with 
	\begin{align}\label{rhsCARE}
	U:=[U_Q,V_A,X_0U_A,X_0U_F],
	\quad D=\diag\left(D_Q,\smb 0&I_{\rank{(U_A)}}\\I_{\rank{(U_A)}}&0\sme,D_F\right),
	\end{align}
	where $\diag$ denotes a block diagonal matrix with the blocks determined by the arguments.
	The correction equation~\eqref{eq:cor-care} now reads as
	\begin{equation} \label{eq:corr2}
	A_{\mathsf{corr}}^*\delta X +\delta XA_{\mathsf{corr}}- \delta XF\delta X + \widehat Q= 0.
	\end{equation}

	It is recommended to perform an optional preprocessing step that aims at reducing the rank of $\hat Q$ further.
	For this purpose, we compute a
	thin QR factorization $U  = Q_UR_U$ followed by a (reordered) spectral decomposition
	\[
	R_U D R_U^* = \begin{bmatrix} S_1 & S_2 \end{bmatrix}  \begin{bmatrix} \Lambda_1 & 0 \\ 0 & \Lambda_2 \end{bmatrix} \begin{bmatrix} S_1 & S_2 \end{bmatrix}^*,
	\]
	such that the diagonal matrix $\Lambda_2$ contains all eigenvalues of magnitude smaller than a prescribed tolerance $\tau_\sigma$.
	Discarding these eigenvalues results in the reduced-rank approximation $\hat Q\approx U D U^*$ with 
	$U \gets Q_U S_1$ and $D \gets \Lambda_1$.
	
	For a large-scale CARE of the form~\eqref{eq:corr2}, with both $\widehat Q$ and $F=BB^*$ of low rank, 
	various numerical methods have been proposed~\cite{BenLP08,Benner2013,simoncini2013two,Sim16,BenBKetal18,BenBKetal18a}. In the following, we focus on the rational Krylov subspace method (RKSM)~\cite{simoncini2013two,Sim16}, but other solvers could be used as well. While these algorithms usually assume $\widehat Q$ to be positive semi-definite, their extension to possibly indefinite $\widehat Q$ poses no major obstacle; see also the discussion in~\cite{LanMS15}.
	RKSM constructs an approximate solution of the form $\delta X_t=V_tYV_t^*$, where 
	$V_t$ contains an orthonormal basis of a rational Krylov subspace $\cR\cK_t(A^*_{\mathsf{corr}},U,\xi)$. The small matrix 
	$Y$ is determined via a Galerkin condition, which comes down to solving the compressed CARE
	\begin{align*}
	\tilde A^*_{\mathsf{corr}}Y+Y\tilde A_{\mathsf{corr}}-Y\tilde FY +  \tilde UD\tilde U^*= 0,\quad \tilde A_{\mathsf{corr}}:=V_t^*A_{\mathsf{corr}}V_t,~\tilde F:=V_t^*FV_t,~\tilde U:=V_t^*U.
	\end{align*}
	for a stabilizing solution $Y$ which can be addressed by direct algorithms for small, dense CAREs~\cite{BenBKetal15a}. Note that the indefinite inhomogeneities $\tilde UD\tilde U^*$ are not an issue for the existence of such stabilizing solution which will then be also indefinite, see, e.g.,~\cite{Wil71}. 

	In practice it can happen that the Hamiltonian matrix associated to the compressed CARE has eigenvalue close to the imaginary axis which can result in inaccurate solutions $Y$. For refining the accuracy of $Y$ we apply a defect correction strategy similar to~\cite{Mehrmann1988a} given by (at most) 2 steps of a Newton's method.
	Algorithm~\ref{alg:rksm_care} gives a basic illustration of
	this method. 
	\begin{algorithm}
		\caption{RKSM for~\eqref{eq:cor-care} with $Q = UDU^*$ and $F = BB^*$}
		\label{alg:rksm_care}
		\begin{algorithmic}[1]
			\Procedure {Low\_rank\_CARE}{$A_{\mathsf{corr}}$, $B$, $U$, $D$}
			\State $V_1=v_1=\texttt{orth}(U)$ \Comment{Orthonormalize $U$ by thin QR decomposition}
			\For {$t=1,2,\ldots$}
			\State $\tilde A^*_{\mathsf{corr}}\leftarrow V_t^*A_{\mathsf{corr}}V_t,~\tilde B\leftarrow V_t^*B,~\tilde U\leftarrow V_t^*U$
			\State $Y\leftarrow$\textsc{Dense\_CARE}$(\tilde A_{\mathsf{corr}},~\tilde B,~\tilde UD\tilde U^*)$
			\State {\bf if} converged {\bf then return} $\delta X_t=V_tYV_t^*$ %{\bf end if}
			\State  Obtain next shift $\xi_t$. \label{line:careshift}
			\State Solve $(A_{\mathsf{corr}}-\xi_t I)^*\tilde v=v_t$ for $\tilde v$. \label{line:linsolve} 
			\State $\tilde v\leftarrow \tilde v-V_t(V_t^*\tilde v)$, $v_{t+1}=\texttt{orth}(\tilde v)$, $V_{t+1}=[V_t,~v_{t+1}]$
			\EndFor
			\EndProcedure
		\end{algorithmic}
	\end{algorithm}
	We refer to the relevant literature~\cite{DruS11,simoncini2013two,Sim16} for implementation details and only comment on some critical steps.
	For selecting the shift parameters in line~\ref{line:careshift} we employ the adaptive procedure from~\cite{DruS11,simoncini2013two,Sim16}. This may result in complex shifts or, more precisely, in complex
	conjugate pairs of shifts. The increased cost of working in complex arithmetic can be largely reduced by using an appropriate implementation, see, e.g.,~\cite{Ruh94b},
	which also  returns a real approximation $\delta X_t$. 
	The shifted linear systems in line~\ref{line:careshift} involve the matrix  $(A_{\mathsf{corr}}-\xi_t I)^* = (A - \xi_t I - FX_0)^{-1}$. If $A$ is sparse, such a system can be solved, e.g., by combining a sparse direct solver for $A - \xi_t I$ with the Sherman-Morrison-Woodbury formula to incorporate the low-rank modification $FX_0$. If $A$ is a HODLR matrix then $A - \xi_t I - FX_0$ is a HODLR matrix as well and solvers for HODLR matrices can be used.
	%In our experiments the matrix $A$ is sparse such that the combination of Sherman-Morrison-Woodbury formula and sparse-direct solvers for the systems defined by  $A-\xi_t I$ was substantially faster. 
	Algorithm~\ref{alg:rksm_care} is terminated once the residual is sufficiently small, that is, $\|R_t\|_2=\|A^*_{\mathsf{corr}}\delta X_t+\delta X_t A^*_{\mathsf{corr}}-\delta X_t F\delta X_t +   UD U^*\|_2\leq\tau_{\mathsf{care}}$
	for some prescribed tolerance $\tau_{\mathsf{care}}>0$.  An efficient way of computing the residual norm $\|R_t\|_2$ is described in~\cite{Sim16}. 
	After termination, it is recommended to perform an optional post-processing step, which aims at reducing the rank of $\delta X_t$, analogous to the rank-reducing procedure for $\hat Q$ described above.
	
	%%%%%%%%%%%%%%%%%%%%%%%%%%%%%%%%%%%%%%%%%%%%%%%%%%%%%%%
	\subsubsection{Extension to generalized CAREs} \label{sec:gCARE}
	
	In this section, we briefly discuss the extension of our low-rank update procedure to the generalized CARE (GCARE), see~\eqref{eq:care}.
	The reference and modified equation take the form
	\begin{equation}\label{eq:ref-gcare}
	A_0^* X_0E_0  +  E_0^*X_0 A_0 - E_0^*X_0F_0X_0E_0 + Q_0 = 0,
	\end{equation} 
	\begin{equation}\label{eq:pert-gcare}
	\begin{split}
	&(A_0+\delta A)^* (X_0+\delta X)(E_0+\delta E)^*  +  (E_0+\delta E)^*(X_0+\delta X) (A_0+\delta A)\\ 
	&-(E_0+\delta E)^*(X_0+\delta X)(F_0+\delta F)(X_0+\delta X)(E_0+\delta E) + (Q_0+\delta Q) = 0,
	\end{split}
	\end{equation}
	where $\delta A,~\delta E, \delta F$ and $\delta Q$ are of low rank, and both $E=E_0+\delta E$ as well as $E_0$ are invertible. By subtracting \eqref{eq:ref-gcare} from \eqref{eq:pert-gcare}, we find that $\delta X$ solves 
	\begin{equation}\label{eq:cor-gcare}
	\begin{split}
	(A- EX_0F)^* \delta XE  + E^*\delta X (A- EX_0F) - E^*\delta XF\delta XE +  \widehat Q= 0
	\end{split}
	\end{equation}
	where $\widehat Q:= \delta Q +\delta A^*X_0E
	+ E^*X_0\delta A- E^*X_0\delta FX_0E+\delta E^*X_0\delta FX_0E_0+E^*X_0\delta FX_0\delta E$ satisfies $\rank{\widehat Q}\leq
	\rank(Q)+2\rank{(\delta A)}+2\rank{(\delta F)}+\min\{\rank{(\delta F)},\rank{(\delta E)}\}$. Similarly as in~\eqref{rhsCARE}, we can write 
	$\widehat Q=U DU^*$ with
	\begin{align*}
	U&:=[U_Q,V_A,E^*X_0U_A,E_0^*X_0U_F,V_E(U_E^*X_0)U_F],\\
	D&=\hat D^*=\diag\left(D_Q,\smb 0&I_{\rank{(U_A)}}\\I_{\rank{(U_A)}}&0\sme,D_F,\smb
	0&I_{\rank{(U_E)}}\\I_{\rank{(U_E)}}&I_{\rank{(U_E)}}\sme\right).
	\end{align*}
	where $\delta E=U_EV_E^*$ with $U_E,V_E\in\R^{n\times \rank{(\delta E)}}$. Again, an optional rank-reducing step for $\hat Q$ is recommended.
	By implicitly working on the equivalent CARE defined by the coefficients $E^{-1}(A- EX_0F)$, $F$, $E^{-*}\widehat Q E^{-1}$, Algorithm~\ref{alg:rksm_care} extends with minor modifications to~\eqref{eq:cor-gcare}.
	%, e.g., . Then,  the linear systems to be solved take the form $(A-FX_0E-\xi_t E)^*\tilde
	%v=E^*v_t$.
	We refer to~\cite{simoncini2013two,Benner2013,Sim16} for further details.
	
	%%%%%%%%%%%%%%%%%%%%%%%%%%%%%%%%%%%%%%%%%%%%%%%%%%%%%%%
	\subsection{Low-rank solution of correction equation for UQME}
	
	The correction equation~\eqref{eq:cor-uqme} features a constant coefficient that has low rank. 
	However, unlike in the case of CARE, we are not aware of existing large-scale solvers tailored to this situation, neither for UQME nor for the modified form~\eqref{eq:cor-uqme}. For example, a fast cyclic reduction iteration proposed in \cite{Bini2013} requires \emph{both} the quadratic and the constant coefficient (that is, the matrices $A$ and $C$ in \eqref{eq:uqme}) to be of low rank. 
	
	In the following, we develop a novel subspace projection method, largely inspired by the existing techniques for CARE described above.
	
	We first discuss the choice of subspaces for our method 
	and consider the Sylvester equation~\eqref{eq:sylv-uqme-corr} for this purpose. In principle, subspace projection methods for Sylvester equations are well understood, but the coefficients of~\eqref{eq:sylv-uqme-corr}  involve the matrix $X$, which depends on the unknown $\delta X$. Solely for the purpose of choosing the subspaces, we replace $X$ with the reference solution $X_0$ and consider
	\begin{equation} \label{eq:simplifiedsylv}
	(X_0+A^{-1}B)\delta X +\delta X X_0=-A^{-1}\widehat{C} 
	\end{equation}
	instead.
	Again, we assume that the low-rank updates are given in factorized form: $\delta A= U_AV_A^*,\delta B= U_BV_B^*$ and $\delta C= U_CV_C^*$. Then, the right-hand sides of \eqref{eq:sylv-uqme-corr} and~\eqref{eq:simplifiedsylv} can be written as $A^{-1}\widehat{C}=UV^*$ with
	\begin{equation}\label{eq:uvuqme}
	U = [A^{-1}U_A, A^{-1}U_B,A^{-1}U_C], \qquad V=[(X_0^*)^2V_A, X_0^*V_B,V_C].
	\end{equation}
	As for CARE, it is recommended to apply a preprocessing step aiming at reducing the rank of $UV^*$. %Equation \eqref{eq:cor-uqme} now reads as
	%\begin{equation}\label{eq:trasf-cor-uqme}
	%(X_0+A^{-1}B)\delta X +\delta XX_0=-UV^*
	%\end{equation}
	Existing solver for Sylvester equations~\cite{Sim16} suggest the use of rational Krylov subspaces with coefficient matrices $X_0+A^{-1}B$, $X^*_0$ and starting vectors $U$, $V$ in order to solve~\eqref{eq:simplifiedsylv}. Specifically, we choose 
	\begin{align} \label{eq:subspacesuqme}
	\mathcal U_t:=\mathcal {RK}_{2t}(X_0+A^{-1}B,U, \mathbf{\pm 1}_t) ,\qquad
	\mathcal V_t:=\mathcal {RK}_{2t}(X_0^*,V, \mathbf{\pm 1}_t),
	\end{align}
	where $\mathbf{\pm 1}_t=[1,-1,\dots,1,-1]^*\in\mathbb R^{2t}$. This particular choice of shift parameters corresponds to the \emph{extended Krylov subspace} \cite{Simoncini2007} for Sylvester equations, adapted to the case in which the spectra of the coefficients are separated by the unit circle instead of the imaginary line. Indeed, we replaced $0$ and $\infty$, the usual choice in the extended Krylov method, with $T(0)=-1$ and $T(\infty)=1$ where $T(z):=-\frac{1+z}{1-z}$ is the Cayley transform.
	
	Suppose now that $U_t,V_t$ contain orthonormal bases of the subspaces defined in~\eqref{eq:subspacesuqme}. To construct an approximate solution $\delta X_t = U_tYV_t^*$ of the \emph{original} equation~\eqref{eq:sylv-uqme-corr}, we impose a Galerkin condition with respect to the  tensorized space $\mathcal U_t\otimes\mathcal V_t$. This implies that the small matrix $Y$ satisfies the non-symmetric algebraic Riccati equation
	\begin{equation}\label{eq:proj-uqme-corr}
	\begin{split}
	Y\widetilde FY +\widetilde A Y &+Y\widetilde D=\widetilde Q,\qquad \\
	\widetilde A=
	U_t^*(X_0+A^{-1}B)V_t,\quad \widetilde F= V_t^*U_t,\quad \widetilde D&= V_t^*X_0^*V_t,\quad\widetilde Q=
	U_t^*UV^*V_t.
	\end{split}
	\end{equation}
	This compressed equation is solved by the \emph{structured doubling algorithm (SDA)} \cite{Guo2006,BinIM12}, see also Algorithm~\ref{alg:sda}. If the projected Hamiltonian \begin{equation} \label{eq:projhamiltonian} 
	\begin{bmatrix}
	\widetilde A& -\widetilde F\\
	-\widetilde Q&-\widetilde D
	\end{bmatrix}\end{equation}
	has an eigenvalue splitting with respect to the unit disc and both equation \eqref{eq:proj-uqme-corr} and its dual equation (the one obtained interchanging $\widetilde F$ and $\widetilde Q$) admit a minimal solution, then SDA converges quadratically to the minimal solution of \eqref{eq:proj-uqme-corr} \cite[Theorem 5.4]{BinIM12}. %\textcolor{red}{??? Please provide the exact theorem in the book by Bini et al. It is hard to match the algorithm with the notation in the book. ???}  %The matrix $\delta X_t = U_tYV_t^*$ is returned as approximation to the
	%solution of~\eqref{eq:cor-uqme}.
	\begin{algorithm}
		\caption{Structured doubling algorithm for  \eqref{eq:proj-uqme-corr}}\label{alg:sda}
		\begin{algorithmic}[1]
			\Procedure{SDA\_NARE}{$\widetilde A,\widetilde D,\widetilde F,\widetilde Q$}
			\State $\begin{bmatrix} S_{11}&S_{12}\\ S_{21}&S_{22}
			\end{bmatrix}
			\gets \begin{bmatrix}I& \widetilde F\\ 0 & -\widetilde A\end{bmatrix}^{-1}\begin{bmatrix}\widetilde D& 0 \\ -\widetilde Q& I\end{bmatrix}$
			\State $E \gets S_{11},\qquad  G = -S_{12},\qquad 
			P = -S_{21},\qquad F = S_{22}$
			\For{$t=1,2,\dots$}	
			\State {\bf if} converged {\bf then return} $P$ %{\bf end if}
			\State $\widetilde G\ \gets I - G\cdot P,\quad\ \ \ \hspace{1pt}  \ \ \ \    \widetilde P \gets I - P\cdot G$ \label{line:gp}
			\State $	E_1 \gets E^{-1}\widetilde G,\qquad\qquad\    F_1 \gets F^{-1}\widetilde P$
			\State  $G\ \gets G + E_1\cdot G\cdot F,\quad \ P \gets P + F_1\cdot P\cdot E$
			\State $E\ \gets E_1\cdot E,\qquad \qquad\ \ F \gets F_1\cdot F$;
			\EndFor 
			\EndProcedure
		\end{algorithmic}
	\end{algorithm}%
	
	The whole procedure for solving~\eqref{eq:sylv-uqme-corr} is summarized in Algorithm~\ref{alg:rk-sda}. %
	\begin{algorithm}
		\caption{Extended Krylov subspace method for~\eqref{eq:sylv-uqme-corr}}\label{alg:rk-sda}
		\begin{algorithmic}[1]
			\Procedure{Low\_rank\_UQME\_corr}{$A,B,X_0,U,V$}
			\Statex \Comment{$U,V \in\R^{n\times r}$ defined as in~\eqref{eq:uvuqme} or pre-processed by a rank-reducing step}
			\State $\widehat A\gets X_0+A^{-1}B$
			\State $U_1 =\texttt{orth}([(\widehat A+I)^{-1}U,(\widehat A-I)^{-1}U])$,\quad $V_1 =\texttt{orth}([V,( X_0^*+I)^{-1}V,(X_0^*-I)^{-1}V])])$
			\For{$t=1,2,\dots$}
			
			\State $\widetilde A\gets U_t^*\widehat AV_t$,\quad $\widetilde D\gets U_t^*X_0^*V^*V_t$,\quad $\widetilde F\gets V_t^*U_t$,\quad  $\widetilde Q\gets U_t^*UV^*V_t$ \label{line:uqmecompressed}
			\State $Y\gets$  \Call{SDA\_NARE}{$\widetilde A,\widetilde D,\widetilde F,\widetilde Q$}
			\State {\bf if} converged {\bf then return} $\delta X_t:=U_tYV_t^*$ %{\bf end if}
			\State Partition $U_t=[U^{(0)},U^{(+)},U^{(-)}]$ such that $U^{(+)},U^{(-)} \in\R^{n\times n}$
			\State  Partition $V_t=[V^{(0)},V^{(+)},V^{(-)}]$ such that $V^{(+)},V^{(-)} \in\R^{n\times n}$
			\State $\widetilde U=[(\widehat A+I)^{-1}U^{(+)},(\widehat A-I)^{-1}U^{(-)}]$,\quad  $\widetilde V=[( X_0^*+I)^{-1}V^{(+)},( X_0^*-I)^{-1}V^{(-)}]$
			\State $\widetilde U\gets \widetilde U-U_tU_t^*\widetilde U,\quad \tilde U=\texttt{orth}(\tilde U), \quad U_{t+1}=[U_t,\tilde U]$
			\State $\widetilde V\gets \widetilde V-V_tV_t^*\widetilde V,\quad \tilde V=\texttt{orth}(\tilde V), \quad V_{t+1}=[V_t,\tilde V]$
			\EndFor 
			\EndProcedure
		\end{algorithmic}
	\end{algorithm}%
	A few remarks concerning the implementation of Algorithms~\ref{alg:sda} and~\ref{alg:rk-sda}:
	\begin{itemize}
		\item Algorithm~\ref{alg:sda} is stopped either when $\min\{\norm{E}_1,\norm{F}_1\}<10^{-13}$ or when a maximum of $30$ iterations is reached. If the projected Hamiltonian~\eqref{eq:projhamiltonian}  has the desired splitting of eigenvalues with respect to the unit circle then Algorithm~\ref{alg:sda}  converges quadratically and is therefore likely to match the convergence condition within $30$ iterations. Otherwise, we move on and consider the next (enlarged) extended Krylov subspaces.
		\item We rely on the \texttt{rktoolbox} \cite{berljafa2014rational} for executing the rational block Arnoldi processes that return the orthonormal bases $U_t$ and $V_t$. We remark that the compressed matrices in line~\ref{line:uqmecompressed} of Algorithm~\ref{alg:rk-sda} do not
		need to be computed explicitly; they can be obtained from the rational Krylov decomposition by adding an artificial final step with an infinite shift, see \cite[page 74]{Guttelthesis} and \cite{Beckermann2009}. %Moreover, one can avoid to build the matrix $X_0+A^{-1}B$ and work directly with the pencil $(AX_0+B) -\lambda A$.
		\item 
		The number of iterations $t$ in Algorithm~\ref{alg:rk-sda} is chosen
		adaptively 
		to ensure that the relation
		\begin{equation}\label{eq:stop}
		\norm{\delta X_t^2+(X_0+A^{-1}B)\delta X_t + \delta X_t X_0 +A^{-1}\widehat C}_2 \leq \tau_{\mathsf{uqme}}
		\end{equation} is satisfied for some tolerance $\tau_{\mathsf{uqme}}$. The artificial final step with an infinite shift mentioned above allows this relation to be verified efficiently, see~\cite{Simoncini2007,Heyouni2010}. 
		%	\textcolor{red}{??? There is a mismatch: The caption of Algorithm~\ref{alg:rk-sda} claims that the algorithm solves the original correction equation, but the residual test is performed for the Sylvester-like equation. Fix one of them. ???}
		
		\item For applying $(\widehat A\pm I)^{-1}$, $(X_0^*\pm I)^{-1}$, 
		LU factorizations of $\widehat A\pm I$ and $X_0^*\pm I$ are computed once before starting the rational block Arnoldi process.
		
		\item
		After termination of Algorithm~\ref{alg:rk-sda}, it is -- once again -- recommended to perform an optional post-processing step that aims at reducing the rank of $\delta X_t$.
	\end{itemize}

	%%%%%%%%%%%%%%%%%%%%%%%%%%%%%%%%%%%%%%%%%%%%%%%%%%%%%%%%%%%%%%%
	\section{Divide-and-conquer methods} \label{sec:dac}
	
	Having an efficient procedure for performing low-rank updates at hand allows us to design divide-and-conquer methods for quadratic matrix equations with rank structured coefficients. For example, suppose that the coefficients of the CARE~\eqref{eq:pert-care} admit the decompositions
	\begin{equation}\label{eq:decomp}
	A=A_0+\delta A,\quad F=F_0+\delta F,\quad Q=Q_0+\delta Q,
	\end{equation}
	where $A_0,F_0,Q_0$ are block diagonal matrices of the same shape and $\delta A,\delta F,\delta Q$ have low rank. This allows us to split \eqref{eq:pert-care} into the correction equation~\eqref{eq:cor-care}, which we solve with Algorithm~\ref{alg:rksm_care}, and the two smaller, decoupled equations associated with the diagonal blocks of $A_0,F_0,Q_0$. If these diagonal blocks again admit a decomposition of the form~\eqref{eq:decomp}, we recursively repeat the splitting. The described strategy easily adapts to the UMQE~\eqref{eq:pert-uqme}.
	
	The storage and manipulation of the low-rank corrections on the various levels of the recursion requires to work with a suitable format,  such as the HODLR format.
	
	\subsection{HODLR matrices}
	
	A HODLR matrix $A\in \mathbb R^{n\times n}$ admits block partition
	\begin{equation} \label{eq:hodler1level}
	A = \begin{bmatrix}
	A_{11}&A_{12}\\
	A_{21}&A_{22}
	\end{bmatrix},
	\end{equation}
	where $A_{12}$, $A_{21}$ have low rank and $A_{11}$, $A_{22}$ are square matrices that again take the form~\eqref{eq:hodler1level}. This splitting is continued recursively  until the diagonal blocks reach a certain minimal block size $n_{\mathsf{min}}$. Usually, the partitioning is chosen such that $A_{11}$, $A_{22}$ have nearly equal sizes. Banded matrices are an important special case of HODLR matrices.

	We say that $A$ has \emph{HODLR rank $k$} if $k$ is the smallest integer such that the ranks of $A_{21}$ and $A_{12}$ in \eqref{eq:hodler1level} are bounded by $k$ at all levels of the recursion.
	If $k$ remains small then $A$ can be stored efficiently by  replacing each off-diagonal block with its low-rank factors.
	The only dense blocks that need to be stored are the diagonal blocks at the lowest level, see Figure~\ref{fig:hodlr}. In turn, the storage of a HODLR matrix requires $O(kn\log n)$ memory.
	\begin{figure}[!ht]
		\centering
		\includegraphics[width=0.8\textwidth]{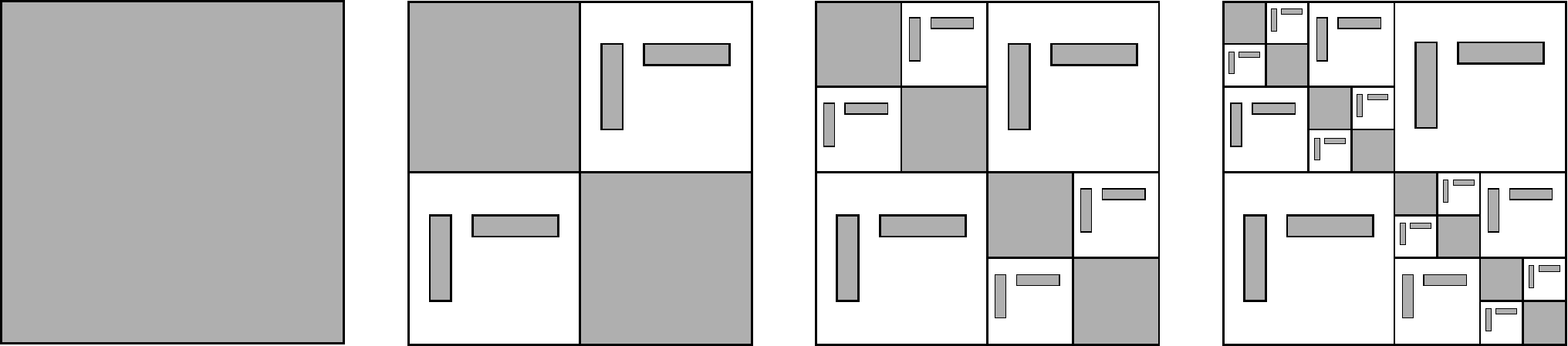}
		\caption{ Image taken from \cite{kressner2017low} describing the HODLR format for different recursion depths.
			%The diagonal sub-blocks filled in gray are dense blocks; the others are stored as low-rank outer products.
		}\label{fig:hodlr}
	\end{figure}
	%%%%%%%%%%%%%%%%%%%%%%%%%%%%%%%%%%%%%%%%%%%%%%%%%%%%%%%%%%%%%
	\subsection{Divide-and-conquer in the HODLR format}
	
	For a CARE~\eqref{eq:pert-care} with HODLR matrices $A, Q$ and a low-rank matrix $F=BB^*$, a divide-and-conquer method can be derived along the lines of the linear case discussed in~\cite{kressner2017low}. Consider 
	\begin{subequations}\label{decompAFQ}
		\begin{align}
		A=\smb
		A_{11}&0\\
		0&A_{22}
		\sme
		+
		\smb
		0&U_1V_2^*\\
		U_1V_2^*&0
		\sme
		=\smb
		A_{11}&0\\
		0&A_{22}
		\sme
		+
		U_AV_A^*,
		\quad U_A=\smb
		U_1&\\&U_2
		\sme,\quad
		V_A=\smb
		V_1&\\&V_2
		\sme
		\end{align}
		and likewise, by exploiting symmetry and low-rank structure  for $Q$ and $F$, the splittings
		\begin{align}
		Q&=\smb
		Q_{11}&0\\
		0&Q_{22}
		\sme
		+
		\smb
		0&U_1U_2^*\\
		U_2U_1^*&0
		\sme
		=\smb
		Q_{11}&0\\
		0&Q_{22}
		\sme
		+
		U_QD_QU_Q^*,
		\quad U_Q=\smb
		U_1&\\&U_2
		\sme,\quad
		D_Q=\smb
		0&I\\I&0
		\sme,\\
		F&=BB^*=\smb
		B_1\\B_2
		\sme
		\smb
		B_1\\B_2
		\sme^*
		=
		U_FU_F^*
		+U_FD_FU_F,
		\quad U_F=\smb
		B_1&\\&B_2
		\sme,\quad
		D_F=\smb
		0&I\\I&0
		\sme.
		\end{align}
	\end{subequations}
	The diagonal blocks $A_{ii},~Q_{ii}$, $i=1,2$ are again HODLR matrices (with the recursion depth reduced by one). After recursively solving the CAREs associated with the diagonal
	blocks, a low-rank approximation to the solution of the correction equation~\eqref{eq:cor-care} is obtained with Algorithm~\ref{alg:rksm_care}.  The resulting procedure is summarized in Algorithm~\ref{alg:dac_care}. 
	As highlighted in the pseudo-code, it is strongly recommended to reduce the ranks of $UDU^*$ in line~\ref{compr_rhs} and of $X_0+\delta X$ in
	line~\ref{compr_sol}. Algorithm~\ref{alg:dac_care} requires the equations associated with the diagonal blocks to admit a unique stabilizing solution at all levels of the recursion.
	% Also the computed low-rank solution $\delta X$ of~\eqref{eq:cor-care} should be
	% recompressed, although this is often already done within Algorithm~\ref{}
	\begin{algorithm}
		\caption{Divide-and-conquer method for CARE with HODLR coefficients}\label{alg:dac_care}
		\begin{algorithmic}[1]
			\Procedure{D\&C\_CARE}{$A,B,Q$}\Comment{Solve $A^*X+XA-XBB^TX+Q=0$}
			\If{$A$ is a dense matrix}
			\State \Return \Call{Dense\_CARE}{$A,B,Q$}
			\Else
			\State Decompose \[A =\begin{bmatrix}
			A_{11}&0 \\ 0&A_{22}
			\end{bmatrix}+U_AV_A^*,\ F=U_FU_F^* + U_FD_FU_F^*,\, Q=\begin{bmatrix}
			Q_{11}& 0\\ 0&Q_{22}
			\end{bmatrix}+U_QD_QU_Q^*\]
			
			~~~with $U_A, V_A, U_F, D_F,U_Q,D_Q$ defined as in~\eqref{decompAFQ}.
			\State $X_{11} \gets $ \Call{D\&C\_CARE}{$A_{11},B_{1},Q_{11}$}
			\State $X_{22} \gets $ \Call{D\&C\_CARE}{$A_{22},B_{2},Q_{22}$}
			\State \label{line:X0} Set $X_0\gets\begin{bmatrix}
			X_{11} & 0 \\ 0&X_{22}
			\end{bmatrix}$
			\State \label{line:largerlowrank} Set $U=[U_Q, V_A, X_0U_A,X_0U_F]$ and 
			$D$ as in~\eqref{rhsCARE}. \Comment{Compression is recommended}\label{compr_rhs}
			\State \label{line:lrupdate} $\delta X\gets $  \Call{low\_rank\_CARE}{$A-(X_0B)B^*,B,U,D$}\Comment{Algorithm~\ref{alg:rksm_care}}
			\State \label{step:hodlrcompress} \Return $X_0+\delta X$ \Comment{Compression is recommended}\label{compr_sol}
			\EndIf
			\EndProcedure
		\end{algorithmic}
	\end{algorithm}
	
	The divide-and-conquer method for a UQME with HODLR matrix coefficients is derived in an analogous manner. The only substantial changes are that the equations associated with the diagonal blocks are solved by \emph{cyclic reduction}~\cite{Bini2009}, see Algorithm~\ref{alg:cr}.
	The resulting procedure is summarized in Algorithm~\ref{alg:dac_uqme}. This algorithm requires that the matrix polynomials associated with the diagonal blocks --- $\lambda^2A_{jj}+\lambda B_{jj}+C_{jj}$ for $j=1,2$ --- maintain the splitting property \eqref{eq:split}, at all levels of the recursion. Similarly as in Algorithm~\ref{alg:dac_care}, compression is recommended in
	lines~\ref{line:1} and~\ref{line:2} of Algorithm~\ref{alg:dac_uqme}.
	\begin{algorithm}
		\caption{Cyclic reduction for UQME}\label{alg:cr}
		\begin{algorithmic}[1]
			\Procedure{Dense\_CR}{$A,B,C$}\Comment{Solve $AX^2+BX+C=0$}
			\State $A^{(0)}\gets A$,  $B^{(0)} \gets B$, $\widehat B^{(0)} \gets B$, $C^{(0)}\gets C$
			\For{$t=0,1,\dots$}	
			\State {\bf if} converged {\bf then return} $-(\widehat B^{(t)})^{-1}B$
			\State $A^{(t+1)}\gets-A^{(t)}(B^{(t)})^{-1}A^{(t)}$	
			\State $B^{(t+1)}\gets B^{(t)}-C^{(t)}(B^{(t)})^{-1}A^{(t)}-A^{(t)}(B^{(t)})^{-1}C^{(t)}$
			\State $\widehat B^{(t+1)}\gets\widehat B^{(t)}-C^{(t)}(B^{(t)})^{-1}A^{(t)}$
			\State $C^{(t+1)}\gets-C^{(t)}(B^{(t)})^{-1}C^{(t)}$.
			\EndFor 
			\EndProcedure
		\end{algorithmic}
	\end{algorithm}
	\begin{algorithm}
		\caption{Divide-and-conquer method for UQME with HODLR coefficients}\label{alg:dac_uqme}
		\begin{algorithmic}[1]
			\Procedure{D\&C\_UQME}{$A,B,C$}\Comment{Solve $AX^2+BX+C=0$}
			\If{$A$ is a dense matrix}
			\State \Return \Call{Dense\_CR}{$A,B,C$}
			\Else
			\State Decompose \[A =\begin{bmatrix}
			A_{11}&0 \\ 0&A_{22}
			\end{bmatrix}+U_AV_A^*,\ B =\begin{bmatrix}
			B_{11}&0 \\ 0&B_{22}
			\end{bmatrix}+U_BV_B^*,\, C=\begin{bmatrix}
			C_{11}& 0\\ 0&C_{22}
			\end{bmatrix}+U_CV_C^*.\]
			\State $X_{11} \gets $ \Call{D\&C\_UQME}{$A_{11},B_{11},C_{11}$}
			\State $X_{22} \gets $ \Call{D\&C\_UQME}{$A_{22},B_{22},C_{22}$}
			\State  Set $X_0\gets\begin{bmatrix}
			X_{11} & 0 \\ 0&X_{22}
			\end{bmatrix}$
			\State  Set $U = [A^{-1}U_A, A^{-1}U_B,A^{-1}U_C]$ and  $V=[(X_0^*)^2V_A, X_0^*V_B,V_C]$  \label{line:1}
			\State  $\delta X\gets $  \Call{Low\_rank\_UQME\_corr}{$A,B,X_0,U,V$}
			%\State  \Return \Call{Compress}{
			\State  \Return $X_0+\delta X$ \label{line:2}
			\EndIf
			\EndProcedure
		\end{algorithmic}
	\end{algorithm}
	\subsubsection{Complexity of divide-and-conquer in the HODLR format}\label{sec:complexity}
	
	The complexity of Algorithms~\ref{alg:dac_care} and~\ref{alg:dac_uqme} critically depends on the convergence of the projection methods (Algorithms~\ref{alg:rksm_care} and~\ref{alg:rk-sda}, respectively) used for solving the correction equations. To a milder extent, it also depends on the numerical methods used for solving the small dense equations associated with the diagonal blocks on the lowest level of the recursion. In order to provide some insights of the computational cost we make the following simplifying assumptions:
	
	\begin{enumerate}[label=(\roman*)]
		\item[(i)]  Algorithm~\ref{alg:rksm_care} and Algorithm~\ref{alg:rk-sda} converge in a constant number of iterations;
		\item[(ii)] solving the dense (unstructured) equations has complexity $\mathcal O(n^3)$;
		\item[(iii)]  the matrix $Q$ in CARE has rank $k$;
		\item[(iv)]  all involved HODLR matrices have HODLR rank $k$ and have a regular partition, that is, $n=2^pn_{\mathsf{min}}$ and the splitting \eqref{eq:hodler1level} always generates equally sized diagonal blocks;
		\item[(v)] the compressions in  Algorithm~\ref{alg:dac_care} and Algorithm~\ref{alg:dac_uqme}  is \emph{not} performed.
	\end{enumerate}
	Under the assumptions stated above, the LU decomposition of an $n\times n$  HODLR matrix requires $\mathcal O(k^2n\log^2(n))$ operations,  while performing forward or backward substitution with a vector is $\mathcal O(kn\log(n))$. A matrix-vector product is $\mathcal O(kn\log(n))$ and all involved matrix-matrix operations are at most $\mathcal O(k^2n\log^2(n))$, see, e.g.,~\cite{Hackbusch2015}.
	
	\begin{paragraph}{CARE.}
		Let $\mathcal C_{\mathsf{care}}(n,k)$ denote the complexity of Algorithm~\ref{alg:dac_care}. 
		Assumption (i) implies that the cost of Algorithm~\ref{alg:rksm_care}, called at  Line~\ref{line:lrupdate}, is $\mathcal O(k^2n\log^2(n))$, because it is dominated by the cost of solving (shifted) linear systems with the matrix $A_{\mathsf{corr}}$.  Assumption (i) also implies that $X_0$, see Line~\ref{line:X0}, has HODLR rank $\mathcal O(k\log(n))$. Because $U_A$ and $U_F$ each have $2k$ columns, the matrix multiplications $X_0U_A$ and $X_0U_F$ at Line~\ref{line:largerlowrank} require $\mathcal O(k^2n\log^2(n))$ operations. 
		Finally, thanks to assumption (ii) we have 
		\begin{equation}\label{eq:complexity}
		\mathcal C_{\mathsf{care}}(n,k)= \begin{cases}
		\mathcal O(n_{\mathsf{min}}^3)&\text{if }n= n_{\mathsf{min}},\\
		\mathcal O(k^2n\log^2(n)) + 2\mathcal C_{\mathsf{care}}(\frac n2,k)&\text{otherwise}.
		\end{cases}
		\end{equation}
		Applying the master theorem \cite{cormen} to \eqref{eq:complexity} yields  $\mathcal C_{\mathsf{care}}(n,k)=\mathcal O(k^2n\log^3(n))$.
	\end{paragraph}
	\begin{paragraph}{UQME.}
		Let $\mathcal C_{\mathsf{uqme}}(n,k)$ denote the complexity of Algorithm~\ref{alg:dac_uqme}. 
		Analogously to CARE,  Assumption (i) implies that Algorithm~\ref{alg:rksm_care} requires $\mathcal O(k^2n\log(n)^2)$ operations and that $X_0$ at Line~\ref{line:X0} has HODLR rank $\mathcal O(k\log(n))$. Therefore, the complexity of Line~\ref{line:largerlowrank} is given by the one of solving $\mathcal O(k)$ linear systems, that is,  $\mathcal O(k^2n\log^2(n))$. In turn, the recurrence relation for $\mathcal C_{\mathsf{uqme}}(n,k)$ is identical with~\eqref{eq:complexity} and hence
		$\mathcal C_{\mathsf{uqme}}(n,k)=O(k^2n\log^3(n))$.
	\end{paragraph}

 	\section{Numerical results} \label{sec:numerical}
 	
 	We now proceed to verify the numerical performance of the divide-and-conquer methods, Algorithms~\ref{alg:dac_care} and~\ref{alg:dac_uqme} from Section~\ref{sec:dac}. Our methods are compared with state-of-the-art iterative algorithms for solving quadratic matrix equations:
 	\begin{itemize}
 		\item structure preserving doubling algorithm (SDA) for CARE \cite{Chu2005},
 		\item cyclic reduction (CR) for UQME \cite{Bini2009} (Algorithm~\ref{alg:cr}).
 	\end{itemize}
 	Both algorithms are well suited for coefficients with hierarchical low-rank structures; we have implemented in HODLR arithmetic using the \texttt{hm-toolbox} \cite{hm-tool}. As indicated in the description of the algorithms, we apply recompression with the threshold $\tau_\sigma=10^{-12}$ in order to keep the ranks under control. Unless stated otherwise, we set the minimal block-size to $n_{\mathsf{min}}=256$ for the representation in the HODLR format. The parameters $\tau_{\mathsf{care}},\tau_{\mathsf{uqme}}$ used in Algorithm~\ref{alg:dac_care} and Algorithm~\ref{alg:dac_uqme}, respectively, for stopping the low-rank iterative solver have been set to $10^{-8}$.
 	
 	All experiments have been performed on a Laptop with a dual-core Intel Core i7-7500U 2.70 GHz CPU, 256KB of level 2 cache, and 16 GB of RAM. The
 	algorithms are implemented in MATLAB and tested under MATLAB2017a, 
 	with MKL BLAS version 11.2.3 utilizing both cores.
 	
 	\subsection{Results for CARE}
 	
 	We will use the following three examples to test the performance of Algorithm~\ref{alg:dac_care} for CARE.
 	\begin{example} \label{example:care1} \rm 
 		This is an academic example of arbitrary size $n$: $A=\tridiag(1,-2,1)$, that is, $A$ is a tridiagonal matrix with $-2$ on the diagonal and $1$ on the sub- and superdiagonal. The matrix $B\in\R^{n\times 2}$ is random with normally distributed entries, $Q=Q_0+ (0.1-\theta)I$, where $Q_0$ is a
 		random symmetric tridiagonal matrix also with normally distributed entries, and $\theta\in\R$ is the smallest eigenvalue of $Q_0$.
 	\end{example}
 	
 	\begin{example}\label{example:care2} \rm 
 		This example is taken from~\cite[Example 3.2 and Section~5.6]{kressner2017low}:
 		\begin{align*}
 		%A&= \tilde A-Z_0Z_0^*BB^*,\quad 
 		A=\smb 0&M\\I&-I\sme,\quad
 		M=\tfrac{1}{4}\tridiag(1,-2,1)-\tfrac{1}{2}(e_1e_1^T + e_n e_n^T), \quad %~M_{n,n}=M_{1,1}=-\tfrac{1}{4},\quad
 		B=\tfrac{1}{4}[e_{n+1},-e_{2n}],\quad Q=I_n,
 		\end{align*}
 		where $e_j$ denotes the $j$th unit vector of appropriate length.
 		Since $A$ is unstable, we use an initial stabilizing solution $X_0:=Z_0Z_0^*$, $Z_0=8\smb -e_n &e_1\\-e_n&e_1\sme$ and consider 
 		the stabilized CARE given by $\tilde A:=A-X_0^*BB^*$ and $\tilde Q:=Q-X_0^*BB^*X_0^*+A^*X_0+X_0A$. Because of the structure of $B$ and $Z_0$, $\tilde A$ is still sparse.
 		All matrices are scaled by $\|A\|_2$ and, to acquire a banded structure, reordered by a perfect shuffle permutation.
 	\end{example}
 	
 	\begin{example} \label{example:care3}\rm 
 		This example is \texttt{carex18} from the CARE benchmark collection~\cite{Abels1999} with tridiagonal $A$ and $E$, $B\in\R^n$, but we set $Q:=I_n$.
 	\end{example}
 	
 	All matrices have been converted into the HODLR format using the \texttt{hmtoolbox}. However, for  the fast solution of the linear systems in Algorithm~\ref{alg:rksm_care} we invoke the original sparse matrices $A$ and $E$ and call a sparse direct solver via MATLAB's ``backslash''.
 	\begin{table}[t]
 		\centering
 		
 		\small
 		\begin{filecontents*}{exp_care1.dat}
1024	2.0641	4.4108e-11	55	6.1118	1.1691e-10	55
2048	4.0234	1.0014e-10	58	21.362	1.824e-09	59
4096	8.3283	5.8513e-10	67	65.418	1.4176e-09	67
8192	18.268	5.6152e-09	67	226	1.33e-08	67
16384	39.709	1.0155e-08	74	670.17	7.4411e-08	76
32768	84.104	5.564e-08	73	2023.8	2.3349e-06	62
\end{filecontents*}
\pgfplotstabletypeset[%column type=l,
 		column type=c,
 		every head row/.style={
 			before row={
 				\toprule
 				& \multicolumn{3}{c|}{Algorithm~\ref{alg:dac_care}} & \multicolumn{3}{c}{\quad SDA}\\
 			},
 			after row = \midrule,
 		},
 		every last row/.style={after row=\bottomrule},
 		sci zerofill,
 		columns={0,1,2,3,4,5,6},
 		columns/0/.style={column name=$n$},
 		columns/1/.style={column name=$\mathrm{Time}$},
 		columns/2/.style={column name=$\mathrm{Res}$},    
 		columns/3/.style={column type/.add={}{|},column name=HODLR rank},
 		columns/4/.style={column name=$\mathrm{Time}$},
 		columns/5/.style={column name=$\mathrm{Res}$},    
 		columns/6/.style={column name=HODLR rank}
 		]{exp_care1.dat}
 		\caption{Execution times (in seconds) and residuals for the divide-and-conquer method and SDA applied to the CARE from Example~\ref{example:care1}. 
 		}
 		\label{tab:care1}
 	\end{table}
 	The results --- compared to those of SDA --- are summarized in the Tables~\ref{tab:care1}--\ref{tab:care3}, where $\mathrm{Res}=\|A^* X  +  X A - XBB^*X + Q\|_2/\|X\|_2$. The reported computing times for both methods clearly reveal that the divide-and-conquer method requires substantially less time than SDA while achieving a similar or even better level of accuracy at the same time.
 	Most of the time in SDA was spent in the numerous HODLR matrix-matrix multiplications and the associated recompression steps after each multiplication. 
 	%\textcolor{red}{??? Is there intermediate rank growth? A few sanity checks that we are not encountering an issue in SDA that could be easily fixed.. ???}
 	%\textcolor{red}{??? Could we say something here why the times are so large for SDA? Presumably most of the time is spent on recompression. Is this within the matrix-matrix multiplication? Is there intermediate rank growth? A few sanity checks that we are not encountering an issue in SDA that could be easily fixed.. ???}

 	For the largest matrices from Example~\ref{example:care1}, $n=32\,768$, we have profiled the computing time spent at the different stages of the divide-and-conquer method. Solving dense CAREs for the diagonal blocks at the lowest level of recursion consumed about 30\% of the total time, while about 50\% was spent on solving the correction equation, CARE~\eqref{eq:cor-care}, by RKSM (Algorithm~\ref{alg:rksm_care}). About 15\% of the time was spent on performing the update $X_0+\delta X$ (line~\ref{compr_sol} in Algorithm~\ref{alg:dac_care}). The work spent on rank compressions was negligible; it consumed less than 1\% of the total time. 
 	Within RKSM, orthonormalization within the Arnoldi method and the solution of the compressed CAREs were the most time consuming steps (totaling approximately 40\% of the time spent on RKSM), followed by the procedure for shift generation (15\%). Due to the sparse, banded structure of $A$, the linear system solves consumed only a very small fraction (about 3\%). We note that the time for solving the correction equation~\eqref{eq:cor-care} could potentially be reduced by employing a different low-rank solver for CAREs. A good candidate for such a solver is the recently proposed RADI method~\cite{BenBKetal18}, which does not rely on Galerkin projections and, hence, does not require solving compressed CAREs\footnote{Preliminary results suggest that replacing Algorithm~\ref{alg:rksm_care} with RADI reduces the time by 10\% on average over all used sizes $n$.}.  
 	We also investigated the effect of reducing the block size $n_{\mathsf{min}}$ from 256 to 128. As expected, this decreased the fraction of computing time spent on diagonal blocks from 30\%  to 10\% but, due to the higher number of occurrences, increased part spent on solving the correction CAREs to about 67\%. The change in the overall time for 
 	Algorithm~\ref{alg:dac_care} was negligible compared to $n_{\mathsf{min}}=256$.

 	\begin{table}[t]
 		\centering
 		
 		\small
 		\begin{filecontents*}{exp_care2.dat}
1024	1.4418	2.3413e-10	10	3.1486	9.5382e-14	5
2048	2.7412	1.2298e-11	9	7.7586	2.5401e-13	5
4096	5.5011	8.031e-12	10	20.432	2.8013e-13	4
8192	11.204	7.5125e-13	10	52.073	2.8784e-13	4
16384	22.209	8.8155e-13	10	131.18	2.8937e-13	3
32768	44.465	1.5567e-13	9	312.57	2.8916e-13	3
\end{filecontents*}
\pgfplotstabletypeset[%column type=l,
 		column type=c,
 		every head row/.style={
 			before row={
 				\toprule
 				& \multicolumn{3}{c|}{Algorithm~\ref{alg:dac_care}} & \multicolumn{3}{c}{\quad SDA}\\
 			},
 			after row = \midrule,
 		},
 		every last row/.style={after row=\bottomrule},
 		sci zerofill,
 		columns={0,1,2,3,4,5,6},
 		columns/0/.style={column name=$n$},
 		columns/1/.style={column name=$\mathrm{Time}$},
 		columns/2/.style={column name=$\mathrm{Res}$},    
 		columns/3/.style={column type/.add={}{|},column name=HODLR rank},
 		columns/4/.style={column name=$\mathrm{Time}$},
 		columns/5/.style={column name=$\mathrm{Res}$},    
 		columns/6/.style={column name=HODLR rank}
 		]{exp_care2.dat}
 		\caption{Execution times (in seconds) and residuals for the divide-and-conquer method and SDA applied to the CARE from Example~\ref{example:care2}.	}
 		\label{tab:care2}
 	\end{table}
 	
 	\begin{table}[t]
 		\centering
 		
 		\small
 		\begin{filecontents*}{exp_care3.dat}
1024	7.5404	1.2047e-07	23
2048	15.009	3.0229e-08	27
4096	29.744	7.576e-09	28
8192	62.005	1.8966e-09	31
16384	128.99	4.7352e-10	28
32768	263.61	1.1875e-10	27
\end{filecontents*}
\pgfplotstabletypeset[%column type=l,
 		column type=c,
 		every head row/.style={
 			before row={
 				\toprule
 				& \multicolumn{3}{c}{Algorithm~\ref{alg:dac_care}}\\
 			},
 			after row = \midrule,
 		},
 		every last row/.style={after row=\bottomrule},
 		sci zerofill,
 		columns={0,1,2,3},
 		columns/0/.style={column name=$n$},
 		columns/1/.style={column name=$\mathrm{Time}$},
 		columns/2/.style={column name=$\mathrm{Res}$},    
 		columns/3/.style={column name=HODLR rank}
 		]{exp_care3.dat}
 		\caption{Execution times (in seconds) and residuals for the divide-and-conquer method and SDA applied to the CARE from Example~\ref{example:care3}. 
 		}
 		\label{tab:care3}
 	\end{table}

 	\subsection{Results for UQMEs from QBD processes} \label{sec:qbd}
 	
 	QBD processes  are discrete-time stochastic processes with a two-dimensional discrete state space. The variables of the state space  are called \emph{level} and \emph{phase}; the transition --- at each time step --- with respect to the level coordinate has at most unit length. We consider models whose state space is isomorphic to $\mathbb N\times \{0,\dots,n-1\}$, that is, we have infinite levels and a finite number of possible phases. Moreover, we assume that the process is \emph{level independent}, i.e. the transition probability depends on the variation of the level but not on its current value.
 	
 	Computing the stationary distribution of a level independent QBD process  amounts to solving a UQME with coefficients corresponding to (possibly shifted) sub-blocks of its transition probability matrix \cite{BiniMarkov}. More specifically, the coefficients of the UQME $AX^2+BX+C=0$ have the properties that $A, B+I,C\in\mathbb R^{n\times n}$ are non-negative and $A+B+C+I$ is stochastic, that is, each row sums to one. 
 	%This framework often ensures the splitting property \eqref{eq:split} with either $\lambda_n=1$ or $\lambda_{n+1}=1$. Under mild assumptions, \eqref{eq:split} is preserved during the recursion in Algorithm~\ref{alg:dac_uqme}.
 	As the following lemma shows, these properties imply --- under some mild additional conditions --- the eigenvalue splitting property~\eqref{eq:split} on every level of recursion in the divide-and-conquer method.
 	\begin{lemma} \label{lemma:uqme}
 		Suppose that $A,B,C$ have the properties stated above and that $\varphi(\lambda)$ has only one eigenvalue on the unit circle, the simple eigenvalue $1$. For some index set $J\subseteq\{1,\dots,n\}$, let $A_J, B_J, C_J$ denote the corresponding principal submatrices of $A,B,C$. Assume that $B_J$ is invertible and $B_J^{-1}(A_J+C_J)$ is irreducible.
 		%and $A_J\circ C_J\neq 0$, where $\circ$ indicates the componentwise product.
 		Then $\varphi_J(\lambda):=\lambda^2A_J+\lambda B_J+C_J$ has the splitting property \eqref{eq:split}.
 	\end{lemma} 
 	\begin{proof}
 		For the moment, let us assume that $A_J+B_J+C_J+I$ is substochastic, that is, $(A_J+B_J+C_J+I)\mathbf e\lneq\mathbf e$, where $\mathbf e$ denotes the vector of all ones and the inequality is understood componentwise. We aim at utilizing the following consequence of 
 		Rouch\'e's theorem for matrix-valued functions~\cite[Theorem 3.2]{Melman2013}: if 
 		\begin{equation} \label{eq:rouchecond}
 		\norm{(\lambda B_J)^{-1}(\lambda^2A_J+C_J)} <1,\qquad \forall |\lambda|=1,
 		\end{equation}
 		holds for an induced norm $\norm{\cdot}$ then $\varphi_J(z)$ has exactly $k$ eigenvalues (counting multiplicities) in the open unit disc and $k$ eigenvalues with modulus greater than $1$, where $k$ denotes the cardinality of $J$. This implies the result of the lemma.
 		
 		Setting $\psi(\lambda):=-B_J^{-1}(\lambda A_J+\lambda^{-1}C_J)$, the condition~\eqref{eq:rouchecond} clearly holds if we can show that the spectral radius $\rho(\psi(\lambda))$  is less than $1$ for every  $\lambda$ on the unit circle. Note that $|\lambda A_J+\lambda^{-1}C_J| \le A_J+C_J$ because $A_J,C_J$ are non-negative. Combined with the fact that $-B_J$ is an M-matrix, which implies $-B_J^{-1}\geq 0$, and the monotonicity of the spectral radius, we obtain
 		\[
 		\rho(\psi(\lambda)) \le \rho(|\psi(\lambda)|) = \rho( | -B_J^{-1} (\lambda A_J+\lambda^{-1}C_J) | ) \le 
 		\rho( -B_J^{-1} (A_J+ C_J) )=\rho(\psi(1)).
 		\]
 		Using $-B_J^{-1}\geq 0$ we also have 
 		\[
 		(A_J+B_J+C_J+I)\mathbf e\lneq\mathbf e\quad\Longrightarrow\quad
 		(A_J+C_J)\mathbf e\lneq -B_J\mathbf e\quad\Longrightarrow\quad \psi(1)\mathbf e\lneq \mathbf e.
 		\]
 		In particular, the matrix $\psi(1)$ is irreducible and substochastic, and by the Perron Frobenius theorem \cite[Theorem 1.5]{Seneta} it has spectral radius strictly less than $1$.
 		
 		It remains to consider the case when $A_J+B_J+C_J+I$ is stochastic. Note that, under this assumption  also the matrix $\psi(1)$ is stochastic. Obviously, the statement of the lemma holds when $A_J = 0$ and $C_J = 0$, so we assume $A_J+ C_J \not=0$ from now on.  Assuming $J = \{1,\ldots, k\}$ after a suitable reordering, we can partition
 		\[
 		\varphi(\lambda) = \begin{bmatrix}
 		\varphi_J(\lambda) & 0 \\
 		\star & \star
 		\end{bmatrix}
 		\]
 		and hence an eigenvalue of $\varphi_J(\lambda)$ is also an eigenvalue of $\varphi(\lambda)$.
 		
 		Let us consider the perturbed matrix polynomial
 		$\varphi_{J,\epsilon}(\lambda):=\lambda^2(A_J-\epsilon E_A)+\lambda B_J+(C_J-\epsilon E_C)$ for $\epsilon >0$ and Boolean matrices $E_A,E_C$ with the sparsity pattern of $A$ and $C$, respectively. Because of $A_J+ C_J \not=0$, the matrix  $A_J-\epsilon E_A+B_J+C_J-\epsilon E_C+I$ is substochastic for $\epsilon$ sufficiently close to $0$. Using again Rouch\'e's theorem, this ensures that $\varphi_{J,\epsilon}(\lambda)$ has the property~\eqref{eq:split}. %, with $\lambda_n<1<\lambda_{n+1}$}.
 		By  continuity, the eigenvalue functions of $\varphi_{J,\epsilon}(\lambda)$  do not cross the unit circle as $\epsilon\to 0$ and, in turn,  $\varphi_J(\lambda)=\lim_{\epsilon\to 0}\varphi_{J,\epsilon}(\lambda)$ has $n$ eigenvalues inside or on the unit circle and $n$ eigenvalues outside or on the unit circle. Because the simple eigenvalue $1$ is the only eigenvalue of $\varphi(\lambda)$
 		on the unit circle and the same property holds for $\varphi_J(\lambda)$, this completes the proof.
 	
 	\end{proof}
 	We remark that the eigenvalue assumption on $\varphi(\lambda)$ in Lemma~\ref{lemma:uqme} can be relaxed to the assumption that $1$ is a simple eigenvalue (admitting possibly other eigenvalues on the unit circle), provided that $B_J^{-1}(A_J+C_J)$ is primitive and $A_J\circ C_J\neq 0$, where $\circ$ indicates the componentwise product.
 	
 	Often, the probabilistic model requires bounded transitions in the phase coordinate as well. This translates into a band structure in the matrices $A,B$ and $C$.
 	For example, in the case of \emph{double QBD processes} (DQBD) \cite{Miyazawa} the coefficients are all tridiagonal, see also Figure~\ref{fig:dqbd}. 
 	\begin{figure}[!ht]
 		\centering
 		\includegraphics[width=0.4\textwidth]{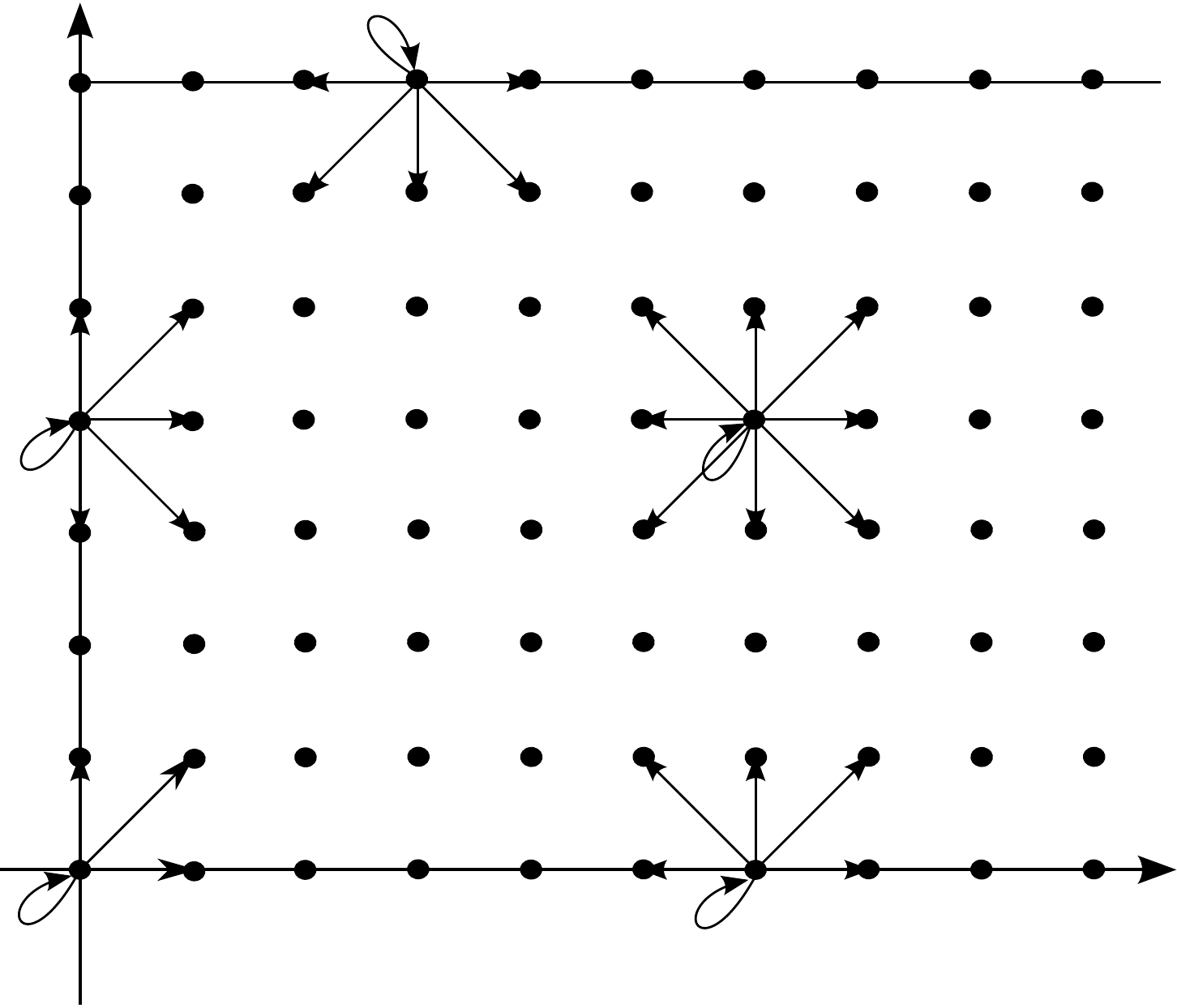}
 		\caption{ Transitions of a double quasi-birth-death process on $\mathbb N \times \{0,\dots n-1\}$.}\label{fig:dqbd}
 	\end{figure}
 	We test Algorithm~\ref{alg:dac_uqme} on instances of DQBD with increasing size $n$. In particular, we choose the entries  of the $3$ central diagonals of $A,B$ and $C$  randomly  from a uniform distribution on $[0,1]$. We divide each row of the three matrices by the corresponding entry in $(A+B+C)\mathbf e$, in order to make $A+B+C$ row stochastic. Finally, we subtract the identity matrix from $B$:
 	\[
 	A=
 	\begin{bmatrix}
 	a_1^{(0)}&a_1^{(1)}\\
 	a_1^{(-1)}&\ddots&\ddots\\
 	& \ddots&\ddots&a_{n-1}^{(1)} \\
 	&&a_{n-1}^{(-1)}&a_{n}^{(0)}
 	\end{bmatrix},
 	\qquad
 	B=
 	\begin{bmatrix}
 	b_1^{(0)}&b_1^{(1)}\\
 	b_1^{(-1)}&\ddots&\ddots\\
 	& \ddots&\ddots&b_{n-1}^{(1)} \\
 	&&b_{n-1}^{(-1)}&b_{n}^{(0)}
 	\end{bmatrix}-I,
 	\]
 	\[
 	C=
 	\begin{bmatrix}
 	c_1^{(0)}&c_1^{(1)}\\
 	c_1^{(-1)}&\ddots&\ddots\\
 	& \ddots&\ddots&c_{n-1}^{(1)} \\
 	&&c_{n-1}^{(-1)}&c_{n}^{(0)}
 	\end{bmatrix}.
 	\]
 	
 	\begin{table}[t]
 		\centering
 		
 		\small
 		\begin{filecontents*}{e_rand_uqme.dat}
1024	1.8363	6.4458e-09	15	3.0023	7.0379e-09	20
2048	3.1951	3.8263e-09	16	10.371	4.3119e-09	18
4096	8.6837	5.0821e-09	17	23.551	6.8213e-09	21
8192	22.265	5.1802e-09	18	68.155	3.8675e-09	20
16384	55.542	6.0955e-09	18	160.28	5.5372e-09	22
32768	137.09	6.6677e-09	18	429.2	8.6149e-09	22
\end{filecontents*}
\pgfplotstabletypeset[%column type=l,
 		column type=c,
 		every head row/.style={
 			before row={
 				\toprule
 				& \multicolumn{3}{c|}{Algorithm~\ref{alg:dac_uqme}} & \multicolumn{3}{c}{\quad CR}\\
 			},
 			after row = \midrule,
 		},
 		every last row/.style={after row=\bottomrule},
 		sci zerofill,
 		columns={0,1,2,3,4,5,6},
 		columns/0/.style={column name=$n$},
 		columns/1/.style={column name=$\mathrm{Time}$},
 		columns/2/.style={column name=$\mathrm{Res}$},    
 		columns/3/.style={column type/.add={}{|},column name=HODLR rank},
 		columns/4/.style={column name=$\mathrm{Time}$},
 		columns/5/.style={column name=$\mathrm{Res}$},    
 		columns/6/.style={column name=HODLR rank}
 		]{e_rand_uqme.dat}
 		\caption{Execution times (in seconds) and residuals for the divide-and conquer-method  and cyclic reduction applied to the example from Section~\ref{sec:qbd}.  
 		}
 		\label{tab:qbd}
 	\end{table}
 	In Table~\ref{tab:qbd} we compare the performance of Algorithm~\ref{alg:dac_uqme} with the method in~\cite{Bini2016} that combines cyclic reduction --- Algorithm~\ref{alg:cr} --- with HODLR arithmetic. Both methods can handle large values for $n$ and return solutions of comparable accuracy, measured in terms of $\mathrm{Res}:=\|AX^2+BX+C\|_2$. However, the divide-and-conquer method provides a significant speed up; it is about $3$ times faster than the competitor for $n\geq 4096$.
 	
 	\subsection{Results for UQMEs from damped mass-spring system} \label{sec:massspring}
 	
 	Another application of UQME is the solution of the quadratic eigenvalue problem $(\lambda ^2A+\lambda B +C)v=0$, arising in the analysis of damped structural systems and vibration problems \cite{Datta10,Lancaster1966,Higham2000b}. After having determined the solution $X$ of \eqref{eq:uqme}, the quadratic eigenvalue problem reduces to two linear eigenvalue problems: the one associated with $X$ and the generalized eigenproblem $(AX+B)v=-\lambda Av$.
 	
 	We repeat the experiments from Section~\ref{sec:qbd} for the UQME associated with a quadratic eigenvalue problem from a damped mass spring system considered in \cite[Example 2]{Tisseur2000}. The $n\times n$ coefficients of the UQME are given by 
 	\[
 	A=I,\qquad 
 	B =\begin{bmatrix}
 	20&-10\\
 	-10&30&-10\\
 	& \ddots&\ddots&\ddots \\
 	&&-10&30&-10\\
 	&&&-10&20
 	\end{bmatrix},\qquad
 	C=
 	\begin{bmatrix}
 	15&-5\\
 	-5&15&-5\\
 	& \ddots&\ddots&\ddots \\
 	&&-5&15&-5\\
 	&&&-5&15
 	\end{bmatrix}.
 	\]
 	The results reported in Table~\ref{tab:spring} confirm the good scalability and accuracy of both methods. The solution exhibits a very low HODLR rank and cyclic reduction needs only 2--3 iterations to converge. As a consequence, cyclic reduction is  faster than Algorithm~\ref{alg:dac_uqme} on larger instances of this example. 
 	\begin{table}[t]
 		\centering
 		
 		\small
 		\begin{filecontents*}{e_mass_uqme.dat}
1024	0.48183	7.7647e-09	2	0.60997	2.3813e-08	4
2048	1.2011	7.7647e-09	2	1.282	2.3813e-08	4
4096	2.9545	7.7646e-09	2	2.971	2.3813e-08	4
8192	7.1772	7.7647e-09	2	6.9496	2.3813e-08	4
16384	18.054	7.7647e-09	2	14.746	2.3813e-08	4
32768	41.138	7.7647e-09	2	34.572	2.3813e-08	4
\end{filecontents*}
\pgfplotstabletypeset[%column type=l,
 		column type=c,
 		every head row/.style={
 			before row={
 				\toprule
 				& \multicolumn{3}{c|}{Algorithm~\ref{alg:dac_uqme}} & \multicolumn{3}{c}{\quad CR}\\
 			},
 			after row = \midrule,
 		},
 		every last row/.style={after row=\bottomrule},
 		sci zerofill,
 		columns={0,1,2,3,4,5,6},
 		columns/0/.style={column name=$n$},
 		columns/1/.style={column name=$\mathrm{Time}$},
 		columns/2/.style={column name=$\mathrm{Res}$},    
 		columns/3/.style={column type/.add={}{|},column name=HODLR rank},
 		columns/4/.style={column name=$\mathrm{Time}$},
 		columns/5/.style={column name=$\mathrm{Res}$},    
 		columns/6/.style={column name=HODLR rank}
 		]{e_mass_uqme.dat}
 		\caption{Execution times (in seconds) and residuals for the divide-and-conquer method  and cyclic reduction  applied to the example from Section~\ref{sec:massspring}.  
 		}
 		\label{tab:spring}
 	\end{table}

	\section{Conclusions}
	
	We have proposed novel Krylov subspace methods for updating the solution of continuous-time algebraic Riccati equations and unilateral quadratic matrix equations whose coefficients are subject to low-rank modifications.  We have provided theoretical insights into the low-rank and stability properties of the solutions to the involved correction equations. This has led us to design novel divide-and-conquer methods for quadratic equations with large-scale coefficients featuring hierarchical low-rank structures. Our methods have linear polylogarithmic complexity and often outperform existing techniques, sometimes significantly. The applications highlighted in this work include quasi-birth--death processes and damped mass-spring systems.

	\textbf{Acknowledgements.} During the larger part of the work on this article, the second author PK was affiliated with the Max Planck Institute for Dynamics of Complex Technical Systems Magdeburg.

	\bibliographystyle{plain}

\end{document}